\documentclass{amsart}

\usepackage{amssymb,amsmath,amscd}

\usepackage[colorlinks,linktocpage]{hyperref}
\hypersetup{linkcolor=[rgb]{0,0,0.715}}
\hypersetup{citecolor=[rgb]{0,0.715,0}}

\newtheorem{thm}{Theorem}[section]
\newtheorem{cor}[thm]{Corollary}
\newtheorem{prop}[thm]{Proposition}
\newtheorem{lem}[thm]{Lemma}
\newtheorem{claim}[thm]{Claim}

\newtheorem{quest}[thm]{Question}

\newtheorem{mainthm}{Theorem}

\theoremstyle{definition}
\newtheorem{defn}[thm]{Definition}

\usepackage{comment}

\theoremstyle{remark}
\newtheorem{rem}[thm]{Remark}

\newcommand{\Ric}{\mathrm{Ric}}
\newcommand{\V}{\mathrm{vol}}

\makeatletter
\let\c@equation\c@thm
\makeatother
\numberwithin{equation}{section}

\title[]{On manifolds with nonnegative Ricci curvature and the infimum of volume growth order $<2$}
\author{Zhu Ye}
\thanks{Supported partially  by National Natural Science Foundation of China [11821101] and [12271372].}
\address[Zhu Ye]{School of Mathematical Sciences, Capital Normal University, Beijing, China.}
\email{2210501006@cnu.edu.cn}
\date{}

\begin{document}
	\begin{abstract}
		we prove two rigidity theorems for open (complete and noncompact) $n$-manifolds $M$ with nonnegative Ricci curvature and the infimum of volume growth order  $<2$. The first theorem asserts that the Riemannian universal cover of $M$ has Euclidean volume growth if and only if $M$ is flat with an $n-1$ dimensional soul. The second theorem asserts that there exists a nonconstant linear growth harmonic function on $M$ if and only if $M$ is isometric to the metric product $\mathbb{R}\times N$ for some compact manifold $N$.
	\end{abstract}
	\maketitle

	\section{Introduction}
	Let $M$ be an open  $n$-manifold with nonnegative Ricci curvature. It is well known that $M$ has at least linear volume growth  (Yau \cite{yau2}) and at most Euclidean volume growth (Bishop volume comparison \cite{bishop}). That is, for some point $p\in M$, we have
\begin{equation*}
	C_1R\leq 	\V (B_R(p)) \leq \omega_nR^n ,\forall R\geq 1,
\end{equation*}
where $C_1>0$ is a constant that may rely on $p$ and $\omega_n$ is the volume of the unit ball in $\mathbb{R}^n$.  We say that $M$ has linear volume growth, if $\V (B_R(p))\leq C_2R $ for some constant $C_2>0$ and any $R\geq 1$. We say that $M$ has Euclidean volume growth or $M$ is noncollapsed, if $\V (B_R(p)) \geq CR^n$ for some constant $C>0$ and any $R\geq 1$. 

In this paper, we will generalize two rigidity theorems that have been established under the condition that $M$ has linear volume growth.

In \cite{Ye}, the author proved the following geometric rigidity:

\begin{thm}\cite{Ye} \label{georigidity}
	Let  $M$ be an open $n$-manifold with $\Ric\geq0$ and linear volume growth. Then the Riemannian universal cover of $M$ is noncollapsed if and only if $M$ is flat with an $n-1$ dimensional soul.
\end{thm}	
As	the first main result of this paper, we obtain an optimal volume growth condition on $M$ such that Theorem \ref{georigidity} still holds. 
  
	\begin{defn}
		We define                    
		\begin{align*}
	\mathrm{IV}(M)	&=\inf\{s>0 \mid \liminf\limits_{R\to\infty} \frac{\V(B_R(p))}{R^s}=0    \},\\
	\mathrm{SV}(M)	&=\sup\{s>0 \mid \limsup\limits_{R\to\infty} \frac{\V(B_R(p))}{R^s}> 0\}.
\end{align*}	
	We will refer to $\mathrm{IV}(M)$ (resp. $ \mathrm{SV}(M)$) as the infimum (resp. supremum) of volume growth order of $M$.
							\end{defn}
By definition, $\mathrm{IV}(M)<k$ if and only if there exists an $\alpha<k$ and a sequence $R_i\to\infty$ such that $\lim\limits_{i\to\infty} \frac{\V (B_{R_i}(p))}{R_i^\alpha}=0.$
	
Our first main result is:
\begin{mainthm} \label{noncollapserigidity}
Let $M$ be an open $n$-manifold with $\Ric\geq 0$ and $\mathrm{IV}(M)<2$. Then the Riemannian universal\ cover of $M$ is noncollapsed if and only if $M$ is flat with an $n-1$ dimensional soul. 
\end{mainthm}

Theorem \ref{noncollapserigidity} is optimal in the sense that $\mathrm{IV}(M)=2$ does not imply flatness. Indeed, we may consider the metric product $M_0=N^2\times F^{n-2}$, where $N$ is a $2$-manifold  with Euclidean volume growth and positive sectional curvature, and $F$ is a compact flat manifold. Then $\mathrm{IV}(M_0)=2$, the universal cover of $M_0$ has Euclidean volume growth, but $M_0$ is not flat. 
 
Assume that $M$ has noncollapsed universal cover. By Theorem \ref{noncollapserigidity}, $\mathrm{IV}(M)<2$ implies that $\mathrm{IV}(M)=\mathrm{SV}(M)=1$. This  motivates the author to propose the following question:
	
\begin{quest} \label{stablevol}
Let $M$ be an open manifold with $\Ric\geq 0$ and noncollapsed universal cover. Is it true that $\mathrm{IV}(M)=\mathrm{SV}(M)\in \mathbb{N}_+$? 
\end{quest}

	Let's illustrate the idea to prove Theorem \ref{noncollapserigidity}. Let $\pi:(\tilde{M},\tilde{p})\rightarrow (M,p)$ be the Riemannian universal cover with deck transformation group $\Gamma$. The volume growth conditions  on $M$ and $\tilde{M}$ guarantee that the order of the orbit growth of $\Gamma$ is strictly larger than  $n-2$  in some scales. This enables us to prove that for some equivariant asymptotic cone (\cite{eGH1,eGH2}) of $(\tilde{M},\tilde{p},\Gamma)$, say $(Y,y,G)$, the orbit $Gy$ has lower Box dimension $\mathrm{dim}_{lb}(Gy)>n-2$ (Proposition \ref{upboxdim}).  Since $\tilde{M}$ has Euclidean volume growth, $Y$ is a metric cone by Cheeger-Colding \cite{almostrigidity}. By Cheeger-Colding splitting theorem, $Y\cong \mathbb{R}^k\times C(Z)$ with $\text{diam}(Z)<\pi$. Now $\mathrm{dim}_{ub}(Gy)>n-2$ forces $Y\cong\mathbb{R}^n$, and thus $\tilde{M}\cong \mathbb{R}^n$ as a corollary of Colding's volume convergence \cite{Colding}. This proves Theorem \ref{noncollapserigidity}. A generalized version of Theorem \ref{noncollapserigidity} is given at the end of the introduction, and we will prove it in Section \ref{proofofA}.  
	
	Our second main result involves linear growth harmonic functions on $M$. Recall that to say a harmonic function $f$ on $M$ has polynomial growth means  $|f|\leq C(d(p,\cdot)+1)^k$ for some $C,k>0$ and $p\in M$. In the particular case $k=1$, we say $f$ has linear growth.
	
	Sormani considered the harmonic functions on  manifolds with linear volume growth, and proved the following theorem: 
	\begin{thm}\cite{Sormaniharmonic} \label{polyharmonicrigidity}Let $M$ be an open $n$-manifold with $\mathrm{Ric}\geq 0$ and linear volume growth.
	If there exists a nonconstant polynomial growth harmonic function on $M$, then  $M$ splits isometrically as $M\cong\mathbb{R}\times N^{n-1}$. 
\end{thm}	
	
	We partially generalize Theorem \ref{polyharmonicrigidity} to the following:  
	\begin{mainthm}\label{harmonicrigidity}
	Let $M$ be an open $n$-manifold such that $\Ric\geq 0$ and $\mathrm{IV}(M)<2$. Then there exists a nonconstant linear growth harmonic function on $M$ if and only if   $M$ splits isometrically as $M\cong\mathbb{R}\times N^{n-1}$ for some compact manifold $N$.   
	\end{mainthm}
		When  $M$ has slow volume growth, Theorem \ref{harmonicrigidity} follows from Section 2 of Li-Tam \cite{LiT}.
		
	Compared with Theorem \ref{polyharmonicrigidity}, we propose the following question:
	\begin{quest}
		Let $M$ be an open $n$-manifold with $\mathrm{Ric}\geq0$ and  $\mathrm{IV}(M)<2$.  If there exists a nonconstant polynomial growth harmonic function on $M$,  is $M$ necessarily isometric to $\mathbb{R}\times N$ for some compact manifold $N$?
	\end{quest}

Let $M$ be complete with $\Ric\geq0$. Given an integer $0\leq k\leq n$, we say that $M$ is $k$-Euclidean at infinity, if any asymptotic cone of $M$ split off an $\mathbb{R}^k$ factor. Note that  $M$ is always $0$-Euclidean at infinity. We
 recall the following result of Cheeger-Colding-Minicozzi:

\begin{thm}\cite{ccm} \label{ccm} Let $M$ be a complete manifold with $\Ric\geq0$. If the space of linear growth harmonic functions on $M$ has dimension $k+1$,  then $M$ is $k$-Euclidean at infinity.
\end{thm}

Theorem \ref{harmonicrigidity} is a by-product of Theorem \ref{ccm} and our research on the relation between the volume growth and the asymptotic cones (with renormalized limit measure) of $M$. 

We will prove the following volume growth gap result:
\begin{mainthm}\label{mainthm}
	Let $M$ be an open manifold with $\Ric\geq 0$. If $M^n$ is $k$-Euclidean at infinity, then 
	$M$ is of one of the following two types:
	
	type I: $M$ has unique asymptotic cone $\mathbb{R}^k$.
	
	type II: every asymptotic cone of $M$ splits as $\mathbb{R}^k\times Z$ with $Z$ noncompact.
	
	If $M$ is of type I, then $\mathrm{IV}(M)=\mathrm{SV}(M)=k$. That is, 
	\begin{equation}\label{voleq1}
		\lim\limits_{R\to\infty}\frac{\V(B_R(p))}{R^{k-\alpha}}=\infty \text{ and }\,	\lim\limits_{R\to\infty}\frac{\V(B_R(p))}{R^{k+\alpha}}=0, \forall\alpha>0.
	\end{equation} 
	
	If  $M$ is of type II, then $\mathrm{IV}(M)\geq k+1$. That is,
	\begin{equation} \label{voleq2}
		\lim\limits_{R\to\infty}\frac{\V(B_R(p))}{R^{k+1-\alpha}}=\infty, \forall\alpha>0.
	\end{equation}	
\end{mainthm}

Theorem \ref{harmonicrigidity} follows from Theorem \ref{ccm} and Theorem \ref{mainthm}.	Let $M$ be an open $n$-manifold with $\Ric\geq0$ and $\mathrm{IV}(M)<2$. By Theorem \ref{ccm}, if there exists a nonconstant linear growth harmonic function on $M$, then $M$ is $1$-Euclidean at infinity. Since $\mathrm{IV}(M)<2$, by Theorem \ref{mainthm}, $M$ can only be $1$-Euclidean of type I, and thus $M\cong \mathbb{R}\times N$ for some compact mianifold  $N$ (see Proposition \ref{uniqueR}).  
	
	The word $k$-Euclidean comes from Cheeger-Colding \cite{CCI}, where a point $p$ in a Ricci limit space is called $k$-Euclidean if any tangent cone at $p$ split off an $\mathbb{R}^k$.

	The proof of Theorem \ref{mainthm} is divided into two steps. 
	
	In Section \ref{k-Euclideanatinfinitytype}, we  prove that if $M$ is $k$-Euclidean at infinity, then either $M$ has unique asymptotic cone $\mathbb{R}^k$ or any asymptotic cone of $M$  is isometric to $\mathbb{R}^k\times Z$ with noncompact $Z$. This is an application of  the critical rescaling technique developed by Pan in \cite{pan1}. See also Pan \cite{pan2,pan3,pan4} for applications of critical rescaling technique to the geometry and topology of open manifolds with nonnegative Ricci curvature.
	
	In Section \ref{volumegrowthandasymcone}, we establish the volume growth estimate (\ref{voleq1}) and (\ref{voleq2}). To achieve this, we shall consider the space of all asymptotic cones with renormalized limit measure and establish the relationship between volume growth and renormalized limit measure. To obtain the volume growth estimate of manifolds with $k$-Euclidean at infinity property of type II, we must use the splitting theorem of $\text{RCD}(0,n)$ spaces, which is proved by Gigli in \cite{gigli}. Compared with the splitting theorem for Ricci limit spaces proved by Cheeger-Colding in \cite{almostrigidity}, an advantage of this version is that it proved that the remaining space after a line splitting is a $\text{CD}(0,n-1)$ space, thus enables us to use volume comparison on it.   
	
Finally, we pointed out that using the $k$-Euclidean at infinity condition, we can generalize Theorem \ref{noncollapserigidity} in the following way:
\begin{mainthm}\label{noncollapsed+Rkinfinity}
Let $M$ be a complete $n$-manifold with $\Ric\geq0$ and noncollapsed universal cover. Assume that $M$ is $k$-Euclidean at infinity.
Then $\mathrm{IV}(M)<k+2$ if and only if $M$ is flat and is isometric to $\mathbb{R}^k\times N^{n-k}$, where $N$ is flat and either closed or open with an $n-k-1$ dimensional soul. 	
\end{mainthm}	
\begin{rem}Note that	Schwarzschild metric on $N^{k+1}=S^{k-1}\times R^2$ is $k$-Euclidean of type I, and thus $\mathrm{IV}(N)=\mathrm{SV}(N)=k$ by Theorem \ref{mainthm}, but $N$ is not flat. This shows the necessity of the noncollapsed condition on the universal cover in Theorem \ref{noncollapsed+Rkinfinity}. 
\end{rem}
Indeed, put $k=0$ in Theorem \ref{noncollapsed+Rkinfinity}, we obtain Theorem \ref{noncollapserigidity}.

\textit{Acknowledgement.}  The author would like to express his sincere gratitude to Professor Xiaochun Rong for his continuous encouragement. The author would like to thank Professor Jiayin Pan for helpful discussions and  suggestions.

\tableofcontents

	\section{Proof of  Theorem \ref{noncollapsed+Rkinfinity} }\label{proofofA}
	In this section, we prove Theorem \ref{noncollapsed+Rkinfinity}. Let $M$ be a complete $n$-manifold with $\Ric\geq0$. Let $\pi:(\tilde{M},\tilde{p})\rightarrow (M,p)$ be the Riemannian universal cover with deck transformation group $\Gamma$. Denote by
	$$\Gamma(R)=\{g\in\Gamma\mid d(\tilde{p},g\tilde{p})\leq R\}.$$
	
	We define the supremum of the orbit growth order of $\Gamma$, denoted by $\mathrm{SO}(\Gamma)$, by
	\begin{equation*}
	\mathrm{SO}(\Gamma)=\sup\{s\mid \limsup\limits_{R\to\infty} \frac{\#(\Gamma(R))}{R^s}>0\}.
	\end{equation*}

We recall the definition of lower Box dimension. For a metric space $X$, a bounded subset $A\subset X$, and an $\epsilon>0$, the $\epsilon$-capacity of $A$ is defined by
$$\mathrm{Cap}(A;\epsilon)=\sup\{k\mid  \text{there are } x_1,\cdots,x_k\in A \text{ such that } d(x_i,x_j)\geq\epsilon, \forall i\neq j \} .$$

The lower Box dimension of $A$, denoted by $\mathrm{dim}_{lb}(A)$, is given by 
$$\mathrm{dim}_{lb}(A)=\liminf\limits_{\epsilon\to0} - \frac{\ln \mathrm{Cap}(A;\epsilon) }
{\ln \epsilon }.$$
We also define the lower box dimension of $X$ as $\mathrm{dim}_{lb}(X)=\sup\limits_{A}\mathrm{dim}_{lb}(A)$, where $A$ run over all bounded subset of $X$.

The following slope lemma \ref{slopelemma}  is inspired by Gromov \cite{Gromov}.

\begin{lem}\label{slopelemma}
	Let $f:[0,\infty)\rightarrow \mathbb{R}$ be a function that is upper bounded on any finite interval: $f|_{[0,R]}\leq C(R)<\infty$ for any  $R>0$. Assume that $f(s_i)\geq ks_i$ for some $k>0$ and a sequence  $s_i\to\infty$. Then for any $l\geq1$, we can find a sequence $r_i\to\infty$ such that
	$$f(r_i)-f(t)> (k-l^{-1})(r_i-t), \forall  t\in [r_i-l, r_i-1] .$$
\end{lem}
\begin{proof}
Assume the Lemma does not hold for some $l\geq 1$. That is, there is an $N(l)>0$, such that for any $r>N(l)$, we have
$$f(r)-f(t_r)\leq (k-l^{-1})(r-t_r)$$
for some $t_r\in [r-l,r-1]$.	Then for $r_0=s_i> N(l)$, we can find a $r_1$ such that  $r_0-l\leq r_1\leq r_0-1$ and $f(r_0)-f(r_1)\leq (k-l^{-1})(r_0-r_1)$. If we still have $r_1>K(l)$, then we can find an $r_2$ such that  $r_1-l\leq r_2\leq r_1-1$ and $f(r_1)-f(r_2)\leq (k-l^{-1})(r_1-r_2)$. Inductively, we can find $r_0=s_i,r_1,\cdots, r_{k_i}$ such that the follow hold:

1. $r_{j}-l\leq r_{j+1}\leq r_j-1$;

2.  $f(r_j)-f(r_{j+1})\leq (k-l^{-1})(r_j-r_{j+1})$ for every $j=0,\cdots, k_i-1$;

3. $r_{k_i-1}>K(l)$, and $r_{k_i}\leq K(l)$.

We have
\begin{align*}
	kr_0 \leq& f(r_0)  \\
	\leq& (f(r_0)-f(r_1))+(f(r_1)-f(r_2))	+\\
	& \cdots+ (f(r_{k_i-1})-f(r_{k_i}))+f(r_{k_i})\\
	\leq & (k-l^{-1})(r_0-r_{k_i})+C(K(l)). 
\end{align*}
When $i\to\infty$, we have $r_0=s_{i}\to\infty$, this leads to a contradiction.

\end{proof}

Lemma \ref{slopelemma} allows us to prove the following general Proposition that relates orbit growth and asymptotic geometry of universal cover. 

\begin{prop} \label{upboxdim}
	There exists an equivariant asymptotic cone of $(\tilde{M},\tilde{p},\Gamma)$, denoted by $(Y,y,G)$, such that the lower box dimension of $Gy$, $\mathrm{dim}_{lb}(Gy)\geq \mathrm{SO}(\Gamma)$.
\end{prop}

\begin{proof}
We	denote  $m=\mathrm{SO}(\Gamma)$. By definition, for every positive integer $i$ and  $m_i=m-\frac{1}{i}$, there exists a sequence $R_{ij}\to\infty$ such that
	$\#(\Gamma(R_{ij}))>e^{m_i} R_{ij}^{m_i}. $
	Assume that $R_{ij}\in [e^{k_{ij}},e^{k_{ij}+1})$, then
	\begin{align*}
		\#(\Gamma(e^{k_{ij}+1}))\geq & \#(\Gamma(R_{ij}))\\
		> & e^{m_i}R_{ij}^{m_i}\\
		\geq &  (e^{k_{ij}+1})^{m_i}.
	\end{align*}
	Put $f(k)=\ln (\#(\Gamma(e^k)) , r_{ij}=k_{ij}+1 $, then $f(r_{ij})>m_ir_{ij}$. Using Lemma \ref{slopelemma}, we find a sequence  $a_{ij}\to\infty$	such that
$$f(a_{ij})-f(t)>(m_i-\frac{1}{i})(a_{ij}-t),\forall t\in [a_{ij}-i, a_{ij}-1].  $$
That is,
\begin{equation} \label{boxeq1}
	\frac{\#(\Gamma(e^{a_{ij}}))}{\#(\Gamma(e^{a_{ij}-t'}))}> e^{(m_i -\frac{1}{i})t'} \text{ for }t'=[1,i] .
\end{equation}

We choose $a_{ij_i}\to\infty$ and put $r_i=e^{a_{ij_i}}$. Let $d$ be the dinstance function on $\tilde{M}$, we denote by $\Gamma_i\tilde{p}$ the orbit $\Gamma\tilde{p}$ equipped with the distance $d_i(g_1\tilde{p},g_2\tilde{p}):= r_i^{-1}d(g_1\tilde{p},g_2\tilde{p})$. For $x\in \Gamma\tilde{p}$ and $r>0$, we put 
$$B^i_r(x)=\{y\in \Gamma_i\tilde{p}\mid d_i(x,y)\leq r \}.$$
For $l\geq 1$, we choose $x_1,x_2,\cdots,x_s\in B_1^i(\tilde{p})$ such that
$$s=\mathrm{Cap}(B^i_1(\tilde{p});e^{-l}), d_i(x_a,x_b)\geq e^{-l} \text{ for any } a\neq b.$$
We have $B^i_1(\tilde{p})\subset \bigcup\limits_{q=1}^s B^i_{e^{-l}}(x_q)$. Since $\#B^i_r(x_a)=\#B^i_r(x_b)$ for any $x_a,x_b\in \Gamma\tilde{p}$, we have  
\begin{equation}\label{boxeq2}
	\mathrm{Cap}(B^i_1(\tilde{p});e^{-l})\geq \frac{\#(B^i_1(\tilde{p}))}{\#(B^i_{e^{-l}}(\tilde{p}))}=\frac{\#(\Gamma(r_i))}{\#(\Gamma(e^{-l}r_i))}. 
\end{equation}

After passing to a subsequence, consider the equivariant point Gromov-Hausdorff convergence (\cite{eGH1,eGH2}) 
	$(r_i^{-1}\tilde{M},\tilde{p},\Gamma)\to (Y,y,G)$.  Denote by $B^\infty_r(y)=\{w \in Gy\mid d_Y(y,w)\leq r\} $. By (\ref{boxeq1}) and (\ref{boxeq2}), we have
	$$\mathrm{Cap}(B^\infty_1(z);e^{-l} )\geq \limsup\limits_{i\to\infty} \mathrm{Cap}(B^{i}_1(\tilde{p});e^{-l})\geq e^{m l}, \text{for any } l\geq 1.$$
	This implies that $\mathrm{dim}_{lb}(Gz)\geq \mathrm{dim}_{lb}(B_1(z)) \geq m$.
\end{proof}
\begin{rem}
	In Proposition 4.2 of \cite{panye}, Pan-Ye proved that for any asymptotic cone $(Y,y,G)$ of $(\tilde{M},\tilde{p},\Gamma)$, $Gy$ has hausdorff dimension $l$, under the condition that $\Gamma$ has stable orbit growth of order $l$: $c_1R^l\leq\#(\Gamma(R))\leq c_2R^l, \forall R\geq1$.  
\end{rem}

	\begin{proof}[Proof of Theorem \ref{noncollapsed+Rkinfinity}]
		Assume that $\mathrm{IV}(M)<k+2$.  By a fundamental domain argument (see Theorem 4 (1.2) in \cite{Ye} and its proof), we have:
		$$\#(\Gamma(2R))\geq \frac{\V (B_R(\tilde{p}))}{\V (B_R(p))}.$$	
		Since $\tilde{M}$ has Euclidean volume growth and $\mathrm{IV}(M)<k+2$, we conclude that $\mathrm{SO}(\Gamma)>n-k-2$. By proposition \ref{upboxdim}, we can find a sequence  $r_i\to\infty$  such that for the equivariant Gromov-Hausdorff convergence  
		$$\begin{CD}
			(r_i^{-1}\tilde{M},\tilde{p},\Gamma) @>GH>> (Y,y,G)\\
			@VV \pi_i V @VV \pi V\\
			(r_i^{-1}M,p) @>GH>> (Z=Y/G,z)
		\end{CD}$$	
		we have $\mathrm{dim}_{lb}(Gy)>n-k-2$.
		Since $M$ is $k$-Euclidean at infinity, we have $(Z,z)\cong (\mathbb{R}^k\times Z',(0^k,z'))$. Since $\pi$ is a submetry, we can write $(Y,y)=(\mathbb{R}^k\times Y',(0^k,y'))$. Note that $\pi$ preserves the $\mathbb{R}^k$ factor, thus we have $G\cdot y\subseteq \{0^k\}\times Y'$.
		
Since $\tilde{M}$ is noncollapsed, $Y$ is a metric cone with vertex $y$. So we can write 
		$$Y=\mathbb{R}^k\times Y'= \mathbb{R}^k \times (\mathbb{R}^m \times C(X))=\mathbb{R}^{k+m}\times C(X)$$
		and $y=(0^k,0^m,v)$, where $C(X)$ does not contain any line and $v$ is the unique vertex of $C(X)$. we have inclusion
	$Gy\subseteq\mathbb{R}^{k+m}\times \{v\}$.
		Together with $Gy\subseteq \{0^k\}\times Y'$, we derive
		$$Gy \subseteq \{0^k\}\times \mathbb{R}^m\times \{v\}.$$
		As $Gy$ has lower box dimension $\mathrm{dim}_{lb}>n-k-2$, we have $m\geq n-k-1$.
		Thus $Y$ splits off an $\mathbb{R}^{n-1}$ factor.  By Cheeger-Colding \cite{CCI}, the set of singular points in $Y$  has codimension at least 2, thus $Y\cong\mathbb{R}^n$. By Colding \cite{Colding},  $\tilde{M}$ itself is isometric to $\mathbb{R}^n$. So $M$ is flat.  Since $M$ is $k$-Euclidean at infinity,  $M$ is isometric to $\mathbb{R}^k\times N^{n-k}$ for some flat manifold $N$. Since $\mathrm{IV}(M)<k+2$, $N$ is either closed or an open flat manifold with an $n-k-1$ dimensional soul.
	\end{proof}

	\section{$k$-Euclidean at infinity}\label{k-Euclideanatinfinitytype}
	In this section, we prove the first part of Theorem \ref{mainthm}:
	\begin{prop} \label{criticalres}
		If $M$ is $k$-Euclidean at infinity, then either $M$ has unique asymptotic cone $\mathbb{R}^k$, or any asymptotic cone $M_{\infty}$ split as $\mathbb{R}^k\times Z$, where $Z$ is noncompact and may rely on $M_\infty$. 
	\end{prop}
	\begin{proof}
		We will prove the following claim:
		\begin{claim}\label{claim1}
			Let $Y_1=\mathbb{R}^k$, $Y_2=\mathbb{R}^k\times Z$, where $Z$ is noncompact, then at most one of them can be an asymptotic cone of $M$.
		\end{claim}	
		Assume  Claim \ref{claim1} holds.  If $(X,x_0)=(\mathbb{R}^k\times K,(0,a))$ with $K\neq \{pt\}$ compact  occurs as an asymptotic cone of $M$, then  a tangent cone of $X$ at $x_0$ is $\mathbb{R}^k\times K'$, where $K'$ is noncompact. Meanwhile, the asymptotic cone of $X$ is $\mathbb{R}^k$. Since any tangent cone at $x_0$ or asymptotic cone of $X$ is an asymptotic cone of $M$, this contradicts  Claim \ref{claim1}. Now Proposition \ref{criticalres} follows from Claim \ref{claim1}.
		
		Now we turn to the proof of Claim \ref{claim1}. Denote by $\Omega$ the set of all asymptotic cones of $M$.
		
		Assume that Claim \ref{claim1} does not hold, that is, both $Y_1$ and $Y_2$ belong to $\Omega$. Then there is an $\epsilon>0$ such that the $\epsilon$ neigberhood of $(Y_1,o_1)$ (according to pointed Gromov-Hausdorff distance) in $\Omega$, $B_{\epsilon}((Y_1,o_1))$ contains no elements like $\mathbb{R}^k \times Z'$ with $Z'$ noncompact.
		
		Choose  $r_i\to0,s_i\to0$ such that 
		\begin{equation*}
			\lim\limits_{i\to\infty} (s_iM, p)=(Y_1,o_1), \lim\limits_{i\to\infty} (r_iM, p)=(Y_2,o_2).
		\end{equation*}
		Write $N_i=r_iM$, $l_i=r_i^{-1}s_i$. Then 
		$$	\lim\limits_{i\to\infty} (N_i, p)=(Y_2,o_2), \lim\limits_{i\to\infty} (l_iN_i, p)=(Y_1,o_1).$$
		After passing to a subsequence, we can always assume that $l_i\to\infty$. 
		
		Set $L_i= \{t\in [1,l_i] \mid d_{pGH}((tN_i,p),(Y_1,o_1) )\leq \frac{\epsilon}{2} \}$, then $l_i\in L_i$ for $i$ large. Let $h_i=\inf L_i\in L_i$.
		
		We claim that $h_i\to\infty$. If not, (after passing subsequence) the limit $\lim\limits_{i\to\infty}(h_iN_i, p) $ will be  $(rY_2,o_2)$ for some $r\geq 1$. Since $rY_2=\mathbb{R}^k\times (rZ)$ with $rZ$ noncompact, we have $d_{pGH}((rY_2,o_2),(Y_1,o_1))\geq \epsilon$. This  contradicts  with $h_i\in L_i$  for $i$ large.
		
		Note that $h_ir_i\leq l_ir_i=s_i\to 0 $, so we can write $$\lim\limits_{i\to\infty} (h_iN_i,p)=\lim\limits_{i\to\infty} (h_ir_iM,p)= (\mathbb{R}^k\times K,o_3)\in \Omega.$$ 
		Since $h_i\in L_i$, $K$ must be compact. Choose $0<c<1$ such that $(\mathbb{R}^k\times (cK),o_3)\in B_{\frac{\epsilon}{3}}(Y_1)$. We have $\lim\limits_{i\to\infty} (ch_iN_i,p)=(\mathbb{R}^k\times (cK),o_3)$,  thus $ch_i\in L_i$ for $i$ large.
		This is impossible since $h_i=\inf L_i$.
	\end{proof}
	The following result is well-known, we put a proof here for readers' convenience:
	\begin{prop}\label{uniqueR}
		If $M$ has unique asymptotic cone $\mathbb{R}$, then $M$ is isometric to $ \mathbb{R}\times N$ for some compact manifold $N$.
	\end{prop}
	\begin{proof}
		Fix a point $p\in M$, then for any sequence $r_i\to \infty$, we have pointed Gromov-Hausdorff convergence
		$$(r_i^{-1}M,p)\xrightarrow{GH} (\mathbb{R},0) .$$
		Let $q_{i,+},q_{i,-}\in r_i^{-1}M$ such that $q_{i,+}\to 1,q_{i,-}\to -1$. Let $\gamma_i$
		be a minimal geodesic connecting $q_{i,-}$ and $q_{i,+}$ and let $a_i$ be a point on $\gamma_i$	that is closest to $p$. Then
		$$\frac{d(q_{i,-},q_{i,+})}{r_i}\to 2,\, \frac{d(a_i,p)}{r_i}\to 0.$$

		Now if $d(p,a_i)< C$ for some constant $C>0$, then $\gamma_i$ converge to a line in $M$. Thus $M\cong \mathbb{R}\times N$ by Cheeger-Gromoll splitting theorem \cite{CG_split}. Since $\mathbb{R}$ is the asymptotic cone of $M$, $N$ must be compact. 
		
		If $d(p,a_i)\to \infty$, we put $s_i=d(p,a_i)$ and  consider the convergence 
		$$(s_i^{-1}M,p)\xrightarrow{GH} (\mathbb{R},0).$$ 
		Since $s_ir_i^{-1}\to0$, passing to subsequence $\gamma_i$ converge to a line $\gamma_\infty$ such that $d(0,\gamma_\infty)=\lim\limits_{i\to\infty}s_i^{-1}d(p,\gamma_i)=1$. Such a line cannot exist on $\mathbb{R}$. So $d(p,a_i)\to \infty$ will not happen.
	\end{proof}

	\section{\label{vol-asym}Volume growth and asymptotic cone} \label{volumegrowthandasymcone}
	To prove the volume growth part of Theorem \ref{mainthm}, we shall study the relation between the renormalized limit measures (see Section 1 of \cite{CCI} for definition) on asymptotic cones and the volume growth of $M$. Let $r_i\to \infty$ such that $\lim\limits_{i\to\infty}(r^{-1}_iM,p)=(X,o,v) $, where $(X,o)$ is the pointed Gromov-Hausdorff limit of $(r_i^{-1}M,p)$ and $v$ is a renormalized limit measure on $X$.
	Then for any $R>0$, we have
	$$v(B_R(o))= \lim\limits_{i\to\infty} \frac{\V(B^{r^{-1}_iM}_R(p))}{\V(B^{r^{-1}_iM}_1(p))}=\lim\limits_{i\to\infty}\frac{\V(B_{r_iR}(p))}{\V(B_{r_i}(p))}, $$ 
	where the existence of the limit is ensured by the construction of $v$. In particular, $v(B_1(o))=1$. Note that an asymptotic cone $(X,o)$ may have different renormalized limit measures on it.

	Let $\Omega$ be the set of all $(X,o,v)$, where
	$(X,o)$ is an asymptotic cone of $M$, and $v$ is a renormalized limit measure on $X$.
	
	For $l>0,l\neq 1$, define $k_l,K_l$ by 
	\begin{align*}
		l^{k_l}=&\min\{v(B_l(o) ) \mid  (Z,o,v)\in \Omega \},\\
		l^{K_l}= &\max \{v(B_l(o) ) \mid  (Z,o,v)\in \Omega \}.
	\end{align*}
	Note that by volume comparison and $v(B_1(o))=1$, we have $0\leq  k_l\leq K_l\leq n$ when $l>1$,  and $0\leq K_l\leq k_l\leq n$ when $0<l<1$.

	\begin{lem}\label{l1} 
		$$l^{k_l}= \liminf\limits_{R\to\infty} \frac{\V(B_{lR}(p))}{\V(B_R(p))}\leq  \limsup\limits_{R\to\infty} \frac{\V(B_{lR}(p))}{\V(B_R(p))}= l^{K_l} .$$
	\end{lem}
	\begin{proof}
		Since 
		$$ \frac{\V(B_{lR}(p))}{\V(B_R(p))}=\frac{\V(B^{R^{-1}M}_l (p))}{\V(B^{R^{-1}M}_1(p))}.$$ 
	\end{proof}
	
	\begin{prop}\label{vol1}
		If $l>1$, then $\forall\alpha>0$, we have
		\begin{equation}\label{eq3}
			\lim\limits_{R\to\infty}\frac{\V(B_R(p))}{R^{K_l+\alpha}}=0,
		\end{equation}
		\begin{equation}\label{eq4}
			\lim\limits_{R\to\infty}\frac{\V(B_R(p))}{R^{k_l-\alpha}}=\infty.
		\end{equation}

		If $0<l<1$, then $\forall\alpha>0$, we have
		\begin{equation*}
			\lim\limits_{R\to\infty}\frac{\V(B_R(p))}{R^{k_l+\alpha}}=0,\,\lim\limits_{R\to\infty}\frac{\V(B_R(p))}{R^{K_l-\alpha}}=\infty.
		\end{equation*}
	\end{prop}
	\begin{proof}
		If $l>1$: By Lemma \ref{l1}, $\forall \epsilon>0$, we can choose $N$ large enough such that $R\geq N\Rightarrow$  
		$$l^{k_l}-\epsilon<\frac{\V(B_{lR}(p))}{\V(B_R(p))}< l^{K_l}+\epsilon. $$
		
		Define function $\phi(R):(N,\infty)\rightarrow \mathbb{N}_{+}$ by $ N\leq \frac{R}{l^{\phi(R)}}< lN$.  We have
		\begin{equation*}
			\V(B_R(p))=\frac{\V(B_R(p))}{\V(B_{\frac{R}{l}}(p))}\cdot \frac{\V(B_{\frac{R}{l}}(p))}{\V(B_{\frac{R}{l^2}}(p))}\cdots \frac{\V(B_{\frac{R}{l^{\phi(R)-1}}}(p))}{\V(B_{\frac{R}{l^{\phi(R)}}}(p))}\cdot \V(B_{\frac{R}{l^{\phi(R)}}}(p)).
		\end{equation*}
		So
		\begin{equation*}
			(l^{k_l}-\epsilon)^{\phi(R)}\V(B_N(p))	< \V(B_R(p))< (l^{K_l}+\epsilon)^{\phi(R)}\V(B_{lN}(p)).
		\end{equation*}
		Now $\forall\alpha>0$, we have
		\begin{align*}
			\frac{\V(B_R(p))}{R^{K_l+\alpha}}<&\frac{(l^{K_l}+\epsilon)^{\phi(R)}\V(B_{lN}(p))}{(Nl^{\phi(R)})^{K_l+\alpha}}\\
			=& \left( \frac{l^{K_l}+\epsilon}{l^{K_l+\alpha}}\right)^{\phi(R)}\frac{\V(B_{lN}(p))}{N^{K_l+\alpha}}.
		\end{align*} 
		Let $\epsilon$ be small enough such that $\frac{l^{K_l}+\epsilon}{l^{K_l+\alpha}}<1$, we obtain (\ref{eq3}).
		
		Similarly, we have
		\begin{align*}
			\frac{\V(B_R(p))}{R^{k_l-\alpha}}>&\frac{(l^{k_l}-\epsilon)^{\phi(R)}\V(B_{N}(p))}{(lNl^{\phi(R)})^{k_l-\alpha}}\\
			=& \left( \frac{l^{k_l}-\epsilon}{l^{k_l-\alpha}}\right)^{\phi(R)}\frac{\V(B_{N}(p))}{(lN)^{k_l-\alpha}},
		\end{align*} 
		Let $\epsilon$ be small enough such that $\frac{l^{k_l}-\epsilon}{l^{k_l-\alpha}}>1$, we obtain (\ref{eq4}).
		
		When $0<l<1$, we obtain from Lemma \ref{l1} that
		$$l^{-K_l}\leq  \liminf\limits_{R\to\infty} \frac{\V(B_{R}(p))}{\V(B_{lR}(p))}\leq  \limsup\limits_{R\to\infty} \frac{\V(B_{R}(p))}{\V(B_{lR}(p))}\leq l^{-k_l} .$$
		The rest of the derivation is similar.

	\end{proof}
	
	\begin{cor}\label{coro1}
		If $(M,p)$ has unique asymptotic cone with unique renormalized limit measure $(X,o,v)$, then there exists a unique $k\geq 0$ such that  
		
		(1) $v(B_l(o))=l^k,\forall l>0$.
		
		(2) the volume growth of $M$ satisfies (\ref{voleq1}).
	\end{cor}
	\begin{proof}
		Fix an $l>0,l\neq 1$, then  $v(B_l(o))=l^{k_l}=l^{K_l}$. We obtain from Proposition \ref{vol1} that $\forall \alpha>0$,
		\begin{equation*}
			\lim\limits_{R\to\infty} \frac{\V(B_R(p))}{R^{k_l+\alpha}}=0,\,\lim\limits_{R\to\infty} \frac{\V(B_R(p))}{R^{k_l-\alpha}}=\infty.
		\end{equation*}
		But the $k_l$ satisfying the above volume growth  is unique, if exists. So $k:=k_l$ is independent of $l$. 
	\end{proof}
	
	Now we can prove the volume growth part of Theorem \ref{mainthm}:
	\begin{prop} If $M$ is $k$-Euclidean at infinity, then:
		
		(1) if $M$ has unique asymptotic cone $\mathbb{R}^k$, then the volume growth of $M$ satisfies (\ref{voleq1}).
		
		(2) if any asymptotic cone is an  $\mathbb{R}^k\times Z$ for some noncompact $Z$, then the volume growth of $M$ satisfies (\ref{voleq2}).
	\end{prop}
	\begin{proof}
		(1) If $M$ has unique asymptotic cone $\mathbb{R}^k$, then the renormalized limit measure $v$ on $\mathbb{R}^k$ is unique and must be the renormalized $k$-dimenisional Lebesgue measure $\frac{1}{\omega_k}\mathcal{L}^k$ on $\mathbb{R}^k$ by Cheeger-Colding \cite{CCI} Proposition 1.35, where $\omega_k=\mathcal{L}^k(B_1(0^k))$.
		So $v(B_l(0^k))=l^k$ for any $l>0$. 
		Inequality (\ref{voleq1}) now follows from Corollary (\ref{coro1}).
		
		(2)  Let $(T,o,v)=(\mathbb{R}^k\times Z, (0^k,z), \,\omega_k^{-1}\mathcal{L}^k\times \mathfrak{m}_{Z})$ be an asymptotic cone of $M$. From  the splitting theorem in RCD$(0,n)$ spaces (\cite{gigli}), we know that $(Z,z,\mathfrak{m}_{Z})$ is a CD$(0,n-k)$ space, hence the volume comparison applies. 
		
		Now $v(B_1(o))=1$ from the definition of renormalized limit measure.  So we have $$\mathfrak{m}_Z(B_1(z))=v(B_1(0^k)\times B_1(z))\geq v(B_1(o))=1.$$
		
		Let $\gamma$   be a ray in  $Z$ starting from $z$, we have
		\begin{align*}
			\mathfrak{m}_Z(B_{2R}(z))\geq&\mathfrak{
				m}_Z(B_1(z))\frac{\mathfrak{m}_Z(B_{R-1}(\gamma(R)))}{\mathfrak{m}_Z(B_{R+1})(\gamma(R))-\mathfrak{m}_Z(B_{R-1}(\gamma(R)))}\\
			\geq&   \frac{(R-1)^{n-k}}{(R+1)^{n-k}-(R-
				1)^{n-k}}\\
			\geq c_nR, \forall R\geq d_n,
		\end{align*}
		where $c_n,d_n>0$ are contants depends only on $n$. Thus 
		
		$$v(B_R(o))\geq \omega_k^{-1}\mathcal{L}^k(B_{\frac{R}{2}}(0))\times \mathfrak{m}(B_{\frac{R}{2}}(z)) \geq c'_n R^{k+1}  ,\forall R\geq d_n.$$
		
		Now write  $v(B_R(o))= R^{k(R)}$. Since 
		$$R^{k(R)}\geq c_n'R^{k+1}   ,\forall R\geq d_n,$$ 
		 for any $\beta>0$, there is a $C(n,\beta)$ such that 
		$$R\geq C(n,\beta)\Rightarrow k(R)> k+1-\beta.$$ 
		We conclude that $k_R\geq k+1-\beta, \forall R\geq C(n,\beta)$  
		(recall that 	$R^{k_R}=\min\{v(B_R(o) ) \mid  (Z,o,v)\in \Omega \}$). The volume growth estimate (\ref{voleq2}) now follows from Proposition \ref{vol1} (\ref{eq4}). 
	\end{proof}

	\bibliographystyle{plain} 
	\bibliography{refs}
\end{document}